\documentclass{birkmult}
 \newtheorem{thm}{Theorem}[section]
 
 \newtheorem{lem}[thm]{Lemma}
 
 \theoremstyle{definition}
 
 \theoremstyle{remark}
 \newtheorem{rem}[thm]{Remark}
 
 \numberwithin{equation}{section}
 
       \usepackage{amsmath}
       \usepackage{amsfonts}
       \usepackage{graphicx}
       \usepackage[noadjust]{cite}

       \newcommand{\R}{\ensuremath{\mathbb{R}}}
       \newcommand{\HQ}{\ensuremath{\mathbb{H}}}

       \newcommand{\be}{\begin{equation}}
       \newcommand{\ee}{\end{equation}}
       
       \newcommand{\G}{\mathcal{G}}
       
  \newcommand{\qi}{\ensuremath{\mbox{\boldmath $i$}}}
  \newcommand{\sqi}{\ensuremath{\mbox{\boldmath $\scriptstyle i$}}}
  \newcommand{\qj}{\ensuremath{\mbox{\boldmath $j$}}}
  \newcommand{\sqj}{\ensuremath{\mbox{\boldmath $\scriptstyle j$}}}
  \newcommand{\qk}{\ensuremath{\mbox{\boldmath $k$}}}
  
  \newcommand{\vect}[1]{\mbox{\boldmath $#1$}}
  \newcommand{\bomega}{\mbox{\boldmath $\omega$}}

\begin{document}
%
%
%
%
%
\submitted{April 22, 2008}
\revised{May 9, 2008}
%
%
%
%
\title[Directional Uncertainty Principle for QFT]
 {Directional Uncertainty Principle  for Quaternion Fourier Transform}
\author[E. Hitzer]{Eckhard M. S. Hitzer}

\address{%
Department of Applied Physics\br
University of Fukui\br
910-8507 Fukui\br
Japan
}

\email{hitzer@mech.fukui-u.ac.jp}

\thanks{In memory of our dear friend Hiroshi Matsushita.}




\date{}
\dedicatory{Soli Deo Gloria}

\begin{abstract}
This paper derives a new directional uncertainty principle
for quaternion valued functions subject to the quaternion Fourier transformation.
This can be generalized to establish directional uncertainty principles in
Clifford geometric algebras with quaternion subalgebras. 
We demonstrate this with the example of a directional spacetime algebra function 
uncertainty principle related to multivector wave packets.
\end{abstract}

\maketitle

\section{Introduction}

The Heisenberg uncertainty principle and quaternions are both fundamental
for quantum mechanics, including the spin of elementary particles. 
The quaternion Fourier transform (QFT)~\cite{TAE:QFT, EH:QFTgen} is used
in image and signal processing. It allows to formulate a component wise
uncertainty principle~\cite{MHHA:QUPcomp, MHAV:WQFT, TB:thesis}.

In Clifford geometric algebras, which generalize real and complex 
numbers and quaternions to higher dimensions, the vector differential
allows to formulate a more general directional uncertainty principle~\cite{MH:CFT,HM:ICCA7}.
The present paper formulates a directional uncertainty principle for 
quaternion functions. As prerequisite for its proof we further
investigate the split of quaternions introduced in~\cite{EH:QFTgen,SG:RQC}. 

Then we show how the generalization of the QFT to a spacetime algebra (STA) 
Fourier transform (SFT)~\cite{EH:QFTgen}
allows us to also generalize the directional QFT uncertainty principle
to STA. There the quaternion split corresponds to the relativistic split of spacetime
into time and space. The split of the SFT corresponds then to analyzing a spacetime
multivector function in terms of left and right travelling multivector wave packets~\cite{EH:QFTgen}.
We will see that the energies of these wave packets determine 
together with two arbitrary spacetime directions
(one spacetime vector and one relativistic wave vector) the resulting uncertainty threshold.

\section{Definition and properties of quaternions $\mathbb{H}$}

\subsection{Basic facts about quaternions}

Gauss, Rodrigues and Hamilton's four-dimensional (4D) quaternion algebra $\HQ$ is defined over $\R$ 
with three imaginary units:

\be
 \qi \qj = -\qj \qi = \qk, \,\,
 \qj \qk = -\qk \qj = \qi, \,\,
 \qk \qi = -\qi \qk = \qj, \,\, 
 \qi^2=\qj^2=\qk^2=\qi \qj \qk = -1.
\label{eq:quat}
\end{equation}
Every quaternion can be written explicitly as
\be
  q=q_r + q_i \qi + q_j \qj + q_k \qk \in \HQ, \quad 
  q_r,q_i, q_j, q_k \in \R,
  \label{eq:aquat}
\end{equation}
and has a \textit{quaternion conjugate} (equivalent to reversion in $Cl_{3,0}^+$)
\be
  \tilde{q} = q_r - q_i \qi - q_j \qj - q_k \qk.
\end{equation}
This leads to the \textit{norm} of $q\in\HQ$
\be
  | q | = \sqrt{q\tilde{q}} = \sqrt{q_r^2+q_i^2+q_j^2+q_k^2},
  \qquad
  | p q | = | p || q |.
\end{equation}
The scalar part of a quaternion is defined as
\be
  Sc(q) = q_r = \frac{1}{2}(q+\tilde{q}).
\ee

\subsection{The $\pm$ split of quaternions}

A convenient \textit{split}~\cite{EH:QFTgen} of quaternions is 
defined by
\begin{gather}
  q = q_+ + q_-, \quad q_{\pm} = \frac{1}{2}(q\pm \qi q \qj).
  \label{eq:pmform}
\end{gather}
Explicitly in real components
$q_r,q_i, q_j, q_k \in \R$ using (\ref{eq:quat}) we get
\be
  q_{\pm} = \{q_r\pm q_k + \qi(q_i\mp q_j)\}\frac{1\pm \qk}{2}
       = \frac{1\pm \qk}{2} \{q_r\pm q_k + \qj(q_j\mp q_i)\}.
  \label{eq:qpm}
\end{equation}
This leads to the following new \textit{modulus identity}
\begin{lem}[Modulus identity]
  \label{lm:modid}
For $q \in \mathbb{H}$
\be
  \label{eq:modid}
  |q|^2 = |q_-|^2 + |q_+|^2.
\ee
\end{lem}
\begin{proof}
Using \eqref{eq:qpm} we get for the right side of \eqref{eq:modid}
\begin{gather}
  |q_-|^2 + |q_+|^2 
  = \frac{1}{2}[(q_r+q_k)^2 + (q_i-q_j)^2 + (q_r-q_k)^2 + (q_i+q_j)^2]
  \nonumber
  \\
  = \frac{1}{2}[2q_r^2 + 2q_k^2 + 2q_i^2 + 2q_j^2] 
  = |q|^2 ,
\end{gather}
because
\be
  \left|\qi\frac{1\pm \qk}{2}\right|^2 = \left|\frac{1\pm \qk}{2}\right|^2 = \frac{1}{2}\,.
\ee
\end{proof}

We can further derive the following useful split product identities.

\begin{lem}[Scalar part of mixed split product]\label{lm:smxprod}
  Given two quaternions $p,q$ and applying the $\pm$ split we get zero for the
  scalar part of the mixed products
  \be
     Sc(p_+\widetilde{q}_-) = 0, \qquad Sc(p_-\widetilde{q}_+) = 0 .
  \ee
\end{lem}
\begin{proof}
We only prove the first identity. The second works analogous. 
\begin{gather}
  Sc(p_+\widetilde{q}_-) 
  = Sc(
       \{ p_r+ p_k + \qi( p_i- p_j)\} \frac{1+ \qk}{2}\frac{1+\qk}{2}
       \{ q_r- q_k - \qi( q_i+ q_j)\} 
      )
  \nonumber \\
  =  Sc(
       \{ p_r+ p_k + \qi( p_i- p_j)\} \frac{\qk}{2}
       \{ q_r- q_k - \qi( q_i+ q_j)\}
       )
  \nonumber \\
  =  \frac{1}{2}Sc(
       \{ p_r+ p_k + \qi( p_i- p_j)\} 
       \{ (q_r- q_k)\qk - \qj( q_i+ q_j)\}
       )
  =  0 , 
\end{gather}
where we used
\be
  \frac{1+ \qk}{2}\frac{1+\qk}{2} =\frac{\qk}{2}.
\ee
\end{proof}

\section{Overview of Fourier transforms (FT)}

We first give a brief overview of some complex and hypercomplex Fourier transforms.

\subsection{Complex, Clifford, quaternion, spacetime Fourier transforms}

The classical complex Fourier transformation (FT) is defined as
\begin{equation}
  \mathcal{F}_{\mathbb{C}} \{ f \} ( \omega ) 
  = \int_{\R} f(x) \,e^{-i \omega x}\, dx,
\end{equation}
where $f\in L^1(\mathbb{R},\mathbb{C})$, and $\omega,x \in \mathbb{R}$.

Brackx et al.~\cite{BDS:CA} extended the Fourier transform to multivector valued 
function-distributions in $Cl_{0,n}$ with compact support. 
A related applied approach for hypercomplex Clifford Fourier 
transformations in $Cl_{0,n}$ was followed by B\"{u}low et. al.~\cite{GS:ncHCFT}.

By extending the classical trigonometric exponential function 
$\exp(j \, \mbox{\boldmath $x$} \ast \mbox{\boldmath $\xi $} )$ (where $\ast $ denotes the 
scalar product of $\mbox{\boldmath $x$}\in \R^m$ with $\mbox{\boldmath $\xi $}\in \R^m$,
$j$ the imaginary unit) 
in~\cite{LMQ:CAFT94,AM:CAFT96}, McIntosh et. al. generalized the classical
Fourier transform. Applied to a function of $m$ real variables this generalized
Fourier transform is holomorphic in $m$ complex variables and its inverse is 
\textit{monogenic}
in $m+1$ real variables, thereby effectively extending the function of $m$ real
variables to a monogenic function of $m+1$ real variables 
(with values in a \textit{complex} Clifford algebra). 
This generalization has significant applications to harmonic analysis,
especially to singular integrals on surfaces in $\R^{m+1}$. 
Based on this approach Kou and Qian obtained a Clifford Payley-Wigner theorem
and derived Shannon interpolation of band-limited functions using the monogenic
sinc function~\cite[and references therein]{TQ:PWT}. 
The Clifford Payley-Wigner theorem also allows to derive left-entire 
(left-monogenic in the whole $\R^{m+1}$) functions from square integrable
functions on $\R^{m}$ with compact support. 

The real $n$-dimensional volume element
$i_n = \mbox{\boldmath $e$}_{1}\mbox{\boldmath $e$}_{2} \ldots \mbox{\boldmath $e$}_{n}$ 
of $\G_n=Cl_{n,0}$ over the field
of the reals $\R$ has been used in \cite{ES:CFTonVF,MH:CFT,EM:CFaUP,HM:ICCA7} 
to construct and apply Clifford Fourier 
transformations for $n=2,3 \, (\rm mod \, 4)$ with kernels 
$\exp(-i_n \vect{x}\ast \bomega),\; \vect{x}, \bomega \in \R^n$. 
This $i_n$ has a clear geometric interpretation. 
Note that $i_n^2=-1$ for $n=2,3 \, (\rm mod \, 4)$.

For $n=3$ the Clifford geometric algebra (GA) FT in $\G_3$ (replacing $i \rightarrow i_3$)
is given by
\begin{equation}
  \label{eq:3DCFT}
  \mathcal{F}_{\G_3}\{f\}(\vec{\omega})
  = \int_{\R^3} {f(\vec{x})
    \,e^{-i_3\vec{\omega}\cdot \vec{x}}\, d^3\vec{x}},
\end{equation}
where $f\in L^1(\R^3,\G_3)$ and $\vec{\omega}, \vec{x} \in \R^3$.
For $n=2,3$ $(\rm mod \, 4)$ the GA FT in $\G_n$ (basically replacing $i \rightarrow i_n$) is
\begin{equation}
  \mathcal{F}_{\G_n}\{f\}(\mbox{\boldmath $\omega$})
  =  \int_{\R^n} f(\mbox{\boldmath $x$})
     \,e^{-i_n\mbox{\boldmath $\omega$}\cdot \mbox{\boldmath $x$}}\, d^n\mbox{\boldmath $x$},
\end{equation}
where $f\in L^1(\R^n,\G_n)$ and $\bomega, \vect{x} \in \R^n$.

Ell \cite{TAE:QFT} defined the quaternion Fourier transform (QFT) 
for application to 2D linear time-invariant systems of PDEs. 
Ell's QFT belongs to the growing family of Clifford Fourier transformations. 
But the left and right placement of the
exponential factors in definition \eqref{eq:EllQFT} distinguishes it. 
Later the QFT was applied extensively to 2D image processing,
including color images \cite{TB:thesis, TAE:QFT, GS:ncHCFT}. 
This spurred research into optimized numerical 
implementations \cite{MF:thesis, PDC:effQFT}. 

The (double sided form of the) QFT in $\mathbb{H}$ (replacing $i \rightarrow \qi, \qj$)
is commonly defined as
\begin{equation}
  \label{eq:EllQFT}
  \mathcal{F}_{\mathbb{H}}\{ f \}(\bomega)
  = \hat{f}(\bomega) 
  = \int_{\R^2} e^{-\sqi x_1\omega_1} f(\vect{x}) \,e^{-\sqj x_2\omega_2} d^2\vect{x},
\end{equation}
where $f\in L^1(\R^2,\mathbb{H})$, $d^2\vect{x} = dx_1dx_2$ and $\vect{x}, \bomega \in \R^2 $.

Ell \cite{TAE:QFT} and others \cite{TB:thesis, CCZ:comQuat} also
investigated related \textit{commutative} hypercomplex Fourier transforms 
like in the commutative subalgebra of $\G_{4}$ with subalgebra basis 
$\{1, \vect{e}_{12}, \vect{e}_{34}, \vect{e}_{1234}\}$, 
\be
  \vect{e}_{12}^2=\vect{e}_{34}^2=-1, \quad \vect{e}_{1234}^2=+1\,\,.
\end{equation}

A higher dimensional generalization~\cite{EH:QFTgen} of the QFT to the spacetime algebra $\G_{3,1}$ is
the spacetime FT (SFT)  
(replacing the quaternion units by $\qi \rightarrow \vect{e}_t$, $\qj \rightarrow i_3$)
will be explained and applied in section \ref{sc:SFT}.

\section{Uncertainty principle}

\subsection{The Heisenberg uncertainty principle in physics}

Photons (quanta of light) scattering off particles to be detected cause 
minimal position momentum uncertainty during detection.
The uncertainty principle has played a fundamental role in the development and 
understanding of quantum physics. 
It is also central for information processing~\cite{JR:HWI}. 

In quantum physics it states e.g. 
that particle momentum and position cannot be simultaneously measured with
arbitrary precision. 
The hypercomplex (quaternion or more general multivector) function $f(\mbox{\boldmath $x$})$ 
would represent the 
spatial part of a separable wave function and its Fourier transform (QFT, CFT, SFT, ...)
$\mathcal{F}\{ f \}(\mbox{\boldmath $\omega$})$ the same wave function
in momentum space (compare~\cite{DL:GAfP,FS:QM,HW:GQM}). 
The variance in space $\R^n$ (in physics often $n=3$) would then be calculated as 
($1\leq k \leq n$)
$$
  (\Delta x_k)^2
  =\int_{\R^n} \hspace{-0.5mm} 
  \langle f(\mbox{\boldmath $x$})
  (\mbox{\boldmath $e$}_k \cdot \mbox{\boldmath $x$})^2
  \tilde{f}(\mbox{\boldmath $x$})\rangle \, d^n \mbox{\boldmath $x$}
  =
  \int_{\R^n} \hspace{-0.5mm} (\mbox{\boldmath $e$}_k \cdot \mbox{\boldmath $x$})^2
  |f(\mbox{\boldmath $x$})|^2 \, d^n \mbox{\boldmath $x$},
$$
where it is customary to set without loss of generality the mean value of 
$\mbox{\boldmath $e$}_k \cdot \mbox{\boldmath $x$}$ to zero~\cite{HW:GQM}. 
The variance in momentum space would be calculated as ($1\leq l \leq n$)
\begin{eqnarray*}
  (\Delta\omega_l)^2  
  & = & \frac{1}{(2\pi)^n} 
  \int_{\R^n} \hspace{-0.5mm} 
  \langle \mathcal{F}\{f\}(\mbox{\boldmath $\omega$})
  (\mbox{\boldmath $e$}_l \cdot \mbox{\boldmath $\omega$})^2
  \widetilde{\mathcal{F}}\{f\}(\mbox{\boldmath $\omega$})\rangle \, 
  d^n \mbox{\boldmath $\omega$} \\
  & = &
  \frac{1}{(2\pi)^n}
  \int_{\R^n}\hspace{-0.5mm}(\mbox{\boldmath $e$}_l\cdot\mbox{\boldmath $\omega$})^2\;
  |\mathcal{F}\{f\}(\mbox{\boldmath $\omega$})|^2 d^n\mbox{\boldmath $\omega$}.
\end{eqnarray*}
Again the mean value of $\mbox{\boldmath $e$}_l \cdot \mbox{\boldmath $\omega$}$ is
customarily set to zero, it merely corresponds to a phase shift~\cite{HW:GQM}.
Using our mathematical units, the position-momentum uncertainty relation of quantum mechanics
is then expressed by (compare e.g. with (4.9) of~\cite[page 86]{FS:QM})
\begin{equation}
  \label{eq:xpUP}
  \Delta x_k \Delta\omega_l = \frac{1}{2}\,\delta_{k,l} F,
\end{equation}
where $\delta_{k,l}$ is the usual Kronecker symbol. Note that we have 
not normalized the squares of the variances by division with 
$F= \int_{\R^n} | f(\mbox{\boldmath $x$})|^2 \; d^n \mbox{\boldmath $x$}$,
therefore the extra factor $F$ on the right side of (\ref{eq:xpUP}).
Further explicit examples
from image processing can be found in~\cite{MF:thesis}.

In general in Fourier analysis such conjugate entities correspond to the variances
of a function and its Fourier transform which cannot both 
be simultaneously sharply localized (e.g.~\cite{JR:HWI,JC:UP}).  
Material on the classical uncertainty principle for the general case of $L^2(\R^n)$ 
without the additional condition $\lim_{|x|\rightarrow \infty}|x|^2|f(x)|=0$
can be found in~\cite{SM:WVT} and~\cite{KG:FTFA}.
Felsberg~\cite{MF:thesis} even notes for two dimensions: \textit{In 2D
however, the uncertainty relation is still an open problem. In~\cite{GK:SPfCV} 
it is stated that
there is no straightforward formulation for the 2D uncertainty relation.}

Let us now briefly define the notion of \textit{directional} uncertainty principle
in Clifford geometric algebra.

\subsection{Directional uncertainty principle in Clifford geometric algebra}

From the view point of Clifford geometric algebra an uncertainty 
principle  gives us information about how the variance of a multivector valued function 
and the variance of its Clifford Fourier transform are related. We can shed the
restriction to the parallel ($k=l$) and orthogonal ($k\neq l$) cases of (\ref{eq:xpUP})
by looking at the $\mbox{\boldmath $x$}\in \R^n$ variance in an arbitrary but
fixed direction $\mbox{\boldmath $a$}\in \R^n$ and 
at the $\mbox{\boldmath $\omega$}\in \R^n$ variance in an arbitrary but fixed direction 
$\mbox{\boldmath $b$}\in \R^n$. 
We are now concerned with
Clifford geometric algebra multivector functions
$f: \R^n \rightarrow \G_n,\, n=2,3$ $(\rm mod \, 4)$ with
GA FT $\mathcal{F}\{ f \}(\mbox{\boldmath $\omega$})$ such that
\be
  \int_{\R^n} | f(\mbox{\boldmath $x$})|^2 \,d^n \mbox{\boldmath $x$} = F <\infty .
\ee
The directional uncertainty principle~\cite{MH:CFT,HM:ICCA7} with arbitrary constant vectors 
\mbox{\boldmath $a$}, \mbox{\boldmath $b$ $\in \mathbb{R}^n$}
means that
\begin{equation}
  \label{eq:CFTdirUP}
  \frac{1}{F}{\int_{\R^n}\hspace{-0.5mm}(\mbox{\boldmath $a$}\cdot\mbox{\boldmath $x$})^2
  |f(\mbox{\boldmath $x$})|^2 \, d^n \mbox{\boldmath $x$} 
  \, 
  {\frac{1}{(2\pi)^n F}\int_{\R^n}\hspace{-0.5mm}(\mbox{\boldmath $b$}\cdot\mbox{\boldmath $\omega$})^2\;
  |\mathcal{F}\{f\}(\mbox{\boldmath $\omega$})|^2 d^n\mbox{\boldmath $\omega$}} }
  \,\geq \,
  \frac{(\mbox{\boldmath $a$}\cdot\mbox{\boldmath $b$})^2}{4}.
\end{equation}
Equality (minimal bound) in \eqref{eq:CFTdirUP} is achieved for optimal  \textit{Gaussian} multivector functions
\begin{equation}
  \label{eqf}
  f(\mbox{\boldmath $x$}) = C_0\; e^{-\alpha\;\mbox{\boldmath $x$}^2},
\end{equation}
$C_0 \in \G_n$ arbitrary const. multivector, $0 < \alpha \in \R.$

So far there seems to be no directional uncertainty principle for quaternion functions
over $\R^2$ subject to the QFT. But there already exists a component wise uncertainty principle.

\subsection{Component wise uncertainty principle for right sided QFT}

For the right sided QFT $\mathcal{F}_{r}\{f\}$: $\mathbb{R}^2 \rightarrow \mathbb{H}$  
of $f \in L^1(\mathbb{R}^2;\mathbb{H})$
given by~\cite{MHHA:QUPcomp, TB:thesis}
\begin{equation}\label{eq:QFTrdef}
  \mathcal{F}_{r}\{f\}(\mbox{\boldmath $\omega$}) 
  = \int_{\mathbb{R}^2} f(\mbox{\boldmath $x$}) e^{-\mbox{\boldmath $i$}\omega_1 x_1}
    e^{-\mbox{\boldmath $j$}\omega_2 x_2}\,
    d^2\mbox{\boldmath $x$},
\end{equation}
where $\mbox{\boldmath $x$}=x_1\mbox{\boldmath $e$}_1 +x_2\mbox{\boldmath
$e$}_2$, 
$\mbox{\boldmath $\omega $}=\omega _1\mbox{\boldmath $e$}_1 +\omega
_2\mbox{\boldmath $e$}_2$, and the quaternion exponential product
$e^{-\mbox{\boldmath $i$}\omega_1 x_1}
e^{-\mbox{\boldmath $j$}\omega_2 x_2}$ is the quaternion Fourier \textit{kernel},
it is possible to establish a component uncertainty principle
($k=1,2$)
\be
  \int_{\R^2}x_k^2|f(\vect{x})|^2d^2\vect{x}\,
  \int_{\R^2}\omega_k^2|\mathcal{F}_{r}\{f\}(\bomega)|^2d^2\bomega
  \geq
  \frac{(2\pi)^2}{4}
  \left\{\int_{\R^2}|f(\vect{x})|^2d^2\vect{x} \right\}^2,
\ee
where $f \in L^2(\mathbb{R}^2; \mathbb{H})$ is now a quaternion-valued signal 
such that both $(1 + | x_k|) f(\mbox{\boldmath $x$}) \in L^2(\mathbb{R}^2;\mathbb{H})$ and
$ \frac{\partial}{\partial x_k} f(\mbox{\boldmath $x$}) \in L^2(\mathbb{R}^2;\mathbb{H})$.
It is further possible to prove~\cite{MHHA:QUPcomp} that equality holds if and only if $f$ 
is a Gaussian quaternion function
\be
   f(\vect{x}) = C_0 e^{-\alpha_1 x_1^2 -\alpha_2 x_2^2}, \qquad  
   \mbox{ const. } C_0 \in \HQ,
   \qquad
   0<\alpha_1,\alpha_2 \in \R.
\ee
\begin{rem}
Now we want to address the very important question:
  Is it also possible to establish a full \textit{directional}
  uncertainty principle for the QFT? 
\end{rem}
In the following we will try to answer this question
after investigating some relevant properties of the \textit{double sided} QFT,
which is slightly different from  \eqref{eq:QFTrdef}.

\section{Quaternion Fourier transform (QFT)}

From now on we only consider the double sided QFT 
$\mathcal{F}\{f\}$: $\mathbb{R}^2 \rightarrow \mathbb{H}$  
of $f \in L^1(\mathbb{R}^2;\mathbb{H})$ in the form
\be
  \label{eq:QFTdef}
  \mathcal{F}\{f\}(\bomega)
  = \hat{f}(\bomega) 
  = \int_{\R^2} e^{-\sqi x_1\omega_1} f(\vect{x}) \,e^{-\sqj x_2\omega_2} d^2\vect{x}.
\end{equation}
Linearity allows to also split up the QFT itself as
\be
  \mathcal{F}\{f\}(\vect{\bomega})
  = \mathcal{F}\{f_- + f_+\}(\vect{\bomega})
  = \mathcal{F}\{f_-\}(\vect{\bomega})
    + \mathcal{F}\{ f_+\}(\vect{\bomega}).
\ee

\subsection{Simple complex forms for QFT of $f_{\pm}$}

\label{th:fpmtrafo}
The QFT of the $f_{\pm}$ split parts of a quaternion function 
$f \in L^2(\R^2,\HQ)$ have simple {complex forms}~\cite{EH:QFTgen}
\begin{equation}
 \hat{f}_{\pm} 
  \stackrel{}{=} \int_{\R^2}
    f_{\pm}e^{-\sqj (x_2\omega_2 \mp x_1\omega_1)}d^2x
  \stackrel{}{=} \int_{\R^2}
    e^{-\sqi (x_1\omega_1 \mp x_2\omega_2)}f_{\pm}d^2x \,\, .
\end{equation}

We can rewrite this free of coordinates as
\begin{equation}
 \hat{f}_{-} 
  \stackrel{}{=} \int_{\R^2}
    f_{-}e^{-\sqj \, \vect{x} \cdot \bomega}d^2x,
 \quad
 \hat{f}_{+} 
  \stackrel{}{=} \int_{\R^2}
    f_{+}e^{-\sqj \, \vect{x} \cdot ({\mathcal U}_{1}\bomega)}d^2x,
\end{equation}
where the reflection ${\mathcal U}_{1}\bomega$ 
changes component $\omega_1 \rightarrow -\omega_1$. That 
this applies to $\omega_1$ is due to the choice of the kernel in \eqref{eq:QFTdef}.

\subsection{Preparations for the full directional QFT uncertainty principle}

Now we continue to lay the ground work for the full directional QFT uncertainty principle.
We begin with the following lemma.

\begin{lem}[Integration of parts]\label{lm:iop}
With the vector differential $\vect{a}\cdot \nabla = a_1 \partial_1+a_2\partial_2$,
with arbitrary constant $\vect{a} \in \R^2$, $g,h\in L^2(\R^2,\HQ)$
\be
  \int_{\R^2} g(\vect{x})[\vect{a}\cdot \nabla h(\vect{x})]d^2\vect{x}
  = \left[ \int_{\R} g(\vect{x})h(\vect{x})d\vect{x}\right]_{a\cdot x = -\infty}^{a\cdot x=\infty}
    - \int_{\R^2} [\vect{a}\cdot \nabla g(\vect{x})]h(\vect{x})d^2\vect{x}.
\ee
\end{lem}
\begin{rem}
The proof of Lemma \ref{lm:iop} works very similar to the proof of integration of parts in~\cite{MH:CFT}.
We therefore don't repeat it here.
\end{rem}
For a quaternion function $f$ and its QFT $\mathcal{F}\{f\}$ itself we also get important modulus identities.
\begin{lem}[Modulus identities]
Due to $|q|^2 = |q_-|^2 + |q_+|^2$ of Lemma \ref{lm:modid} 
we get for $f:\R^2\rightarrow \HQ$ the following identities
\be
  |f(\vect{x})|^2 = |f_-(\vect{x})|^2 + |f_+(\vect{x})|^2,
\ee
\be
  |\mathcal{F}\{f\}(\bomega)|^2 
  = |\mathcal{F}\{f_-\}(\bomega)|^2 
  + |\mathcal{F}\{f_+\}(\bomega)|^2.
\ee
\end{lem}

We further establish formulas for the vector differentials of the QFTs of the split function parts
$f_-$ and $f_+$.
\begin{lem}[QFT of vector differentials]\label{lm:QFTvdiff}
Using the split $f=f_-+f_+$ we get the QFTs of the split parts. Let $\vect{b}\in \R^2$ 
be an arbitrary constant vector. 
\be
  \mathcal{F}\{\vect{b}\cdot\nabla f_-\}(\bomega) 
  = \vect{b}\cdot \bomega \mathcal{F}\{f_-\}(\bomega) \,\qj,
\ee
\be
    \mathcal{F}\{\vect{b}\cdot\nabla f_+\}(\bomega) 
  = \vect{b}\cdot ({\mathcal U}_{1}\bomega) 
    \,\mathcal{F}\{f_+\}(\bomega) \,\qj,
\ee
\be
    \mathcal{F}\{({\mathcal U}_{1}\vect{b})\cdot\nabla f_+\}(\bomega) 
  = \vect{b}\cdot \bomega \,\mathcal{F}\{f_+\}(\bomega) \,\qj .
\ee
\end{lem}
\begin{rem}
The previously explained \textit{complex forms} of the QFTs
make the proof of Lemma \ref{lm:QFTvdiff} very similar to the case $n=2$ of~\cite{HM:ICCA7}.
Noncommutativity must be duly taken into account!
\end{rem}

\begin{lem}[Schwartz inequality]\label{lm:Swineq}
Two quaternion functions $g,h \in L^2(\R^2,\HQ)$ obey the following
Schwartz inequality
\begin{gather}
  \int_{\R^2}|g(\vect{x})|^2d^2\vect{x}
  \int_{\R^2}|h(\vect{x})|^2d^2\vect{x}
  \nonumber
  \\
  \geq
  \frac{1}{4}\left[ \int_{\R^2}g(\vect{x})\tilde{h}(\vect{x}) 
         + h(\vect{x})\tilde{g}(\vect{x}) d^2\vect{x}
    \right]^2
    \nonumber
    \\
  = \left[ \int_{\R^2} Sc(g(\vect{x})\tilde{h}(\vect{x}))
         d^2\vect{x}
    \right]^2.
\end{gather}
\end{lem}
\begin{rem}
The proof of Lemma \ref{lm:Swineq} can be based on the following inequality ($0 < \epsilon \in \R$)
\be
  \int_{\R^2} [g(\vect{x}) + \epsilon h(\vect{x})]
              [g(\vect{x}) + \epsilon h(\vect{x})]^{\sim} d^2\vect{x} \geq 0.
\ee
\end{rem}

\section{Directional QFT uncertainty principle}

In this section we state the directional QFT uncertainty principle
and prove it step by step.
\begin{thm}[Directional QFT UP]
For two arbitrary constant vectors 
$\vect{a}= a_1 \vect{e}_1 + a_2 \vect{e_2}\in \R^2, \vect{b}= b_1 \vect{e}_1 + b_2 \vect{e_2} \in \R^2$ (selecting two 
directions), and $f\in L^2(\R^2,\HQ)$, $|\vect{x}|^{1/2}f \in L^2(\R^2,\HQ)$
we obtain 
\begin{equation}
 \int_{\R^2} (\vect{a}\cdot \vect{x})^2 |f(\vect{x})|^2 d^2\vect{x}
 \int_{\R^2} (\vect{b}\cdot \bomega )^2  |\mathcal{F}\{f\}(\bomega )|^2 d^2\bomega
 \geq \frac{(2\pi )^2}{4}
      \left[ (\vect{a} \cdot \vect{b})^2 F_-^2  
             + (\vect{a} \cdot \vect{b}^{\prime})^2 F_+^2 \right] ,
\end{equation}
with the energies
\be
   F_{\pm} =  \int_{\R^2} |f_{\pm}(\vect{x})|^2 d^2 \vect{x},
   \qquad 
   \vect{b}^{\prime} = {\mathcal U}_{1}\vect{b} = -b_1 \vect{e}_1 + b_2 \vect{e_2}\in \R^2.
\ee
\end{thm}

\begin{proof}
We now prove the directional QFT uncertainty principle by the following direct calculation.

\begin{align*}
   \int_{\R^2} (\vect{a}\cdot \vect{x})^2 &|f(\vect{x})|^2 d^2\vect{x}
   \int_{\R^2} (\vect{b}\cdot \bomega )^2  |\mathcal{F}\{f\}(\bomega )|^2 d^2\bomega
   \\  
   \stackrel{\pm\mbox{split}}{=} 
   &\left[\int_{\R^2} (\vect{a}\cdot \vect{x})^2 |f_-(\vect{x})|^2 d^2\vect{x}
   +\int_{\R^2} (\vect{a}\cdot \vect{x})^2 |f_+(\vect{x})|^2 d^2\vect{x}
   \right]
   \\
   &\left[\int_{\R^2} (\vect{b}\cdot \bomega )^2  |\mathcal{F}\{f_-\}(\bomega )|^2 d^2\bomega
   + \int_{\R^2} (\vect{b}\cdot \bomega )^2  |\mathcal{F}\{f_+\}(\bomega )|^2 d^2\bomega
   \right]
   \\
   \stackrel{\mbox{Lem.\ref{lm:QFTvdiff}}}{=}
   &{\Big[}\ldots {\Big]}
   \left[\int_{\R^2} |\mathcal{F}\{(\vect{b}\cdot \nabla )f_-\}(\bomega )(-\qj)|^2 d^2\bomega
   + \int_{\R^2}  |\mathcal{F}\{(\vect{b}^{\prime}\cdot \nabla )f_+\}(\bomega )(-\qj)|^2 d^2\bomega
   \right]
   \\
   \stackrel{\mbox{Parsev.}}{=}
   &\left[\int_{\R^2} |(\vect{a}\cdot \vect{x})f_-(\vect{x})|^2 d^2\vect{x}
   +\int_{\R^2} |(\vect{a}\cdot \vect{x})f_+(\vect{x})|^2 d^2\vect{x}
   \right]
   \\
   &(2\pi)^2\left[\int_{\R^2} |(\vect{b}\cdot \nabla )f_-(\vect{x} )|^2 d^2\vect{x}
   + \int_{\R^2}  |(\vect{b}^{\prime}\cdot \nabla )f_+(\vect{x} )|^2 d^2\vect{x}
   \right]
   \\
   = \,\,\, &(2\pi)^2\left[
   \int_{\R^2} |(\vect{a}\cdot \vect{x})f_-(\vect{x})|^2 d^2\vect{x}
   \int_{\R^2} |(\vect{b}\cdot \nabla )f_-(\vect{x} )|^2 d^2\vect{x}
   \right.
   \\
   &+
   \int_{\R^2} |(\vect{a}\cdot \vect{x})f_+(\vect{x})|^2 d^2\vect{x}
   \int_{\R^2}  |(\vect{b}^{\prime}\cdot \nabla )f_+(\vect{x} )|^2 d^2\vect{x}
   \\
   &+\int_{\R^2} |(\vect{a}\cdot \vect{x})f_-(\vect{x})|^2 d^2\vect{x}
   \int_{\R^2}  |(\vect{b}^{\prime}\cdot \nabla )f_+(\vect{x} )|^2 d^2\vect{x}
   \\
   &\left.
   +\int_{\R^2} |(\vect{a}\cdot \vect{x})f_+(\vect{x})|^2 d^2\vect{x}
   \int_{\R^2} |(\vect{b}\cdot \nabla )f_-(\vect{x} )|^2 d^2\vect{x}
   \right]
   \\
   \stackrel{\mbox{Schwartz}}{\leq}
   &(2\pi)^2\left[
   \left\{ \int_{\R^2} 
      Sc\left(\vect{a}\cdot \vect{x}f_- \, \vect{b}\cdot \nabla \widetilde{f}_-\right) d^2\vect{x}
   \right\}^2 
   +
   \left\{ \int_{\R^2} 
      Sc\left(\vect{a}\cdot \vect{x}f_+ \, \vect{b}^{\prime}\cdot \nabla \widetilde{f}_+\right) 
      d^2\vect{x}
      \right\}^2  
   \right.
   \\
   &+
   \Big\{ \int_{\R^2} 
      \underbrace{Sc\left(\vect{a}\cdot \vect{x}f_- \, \vect{b}^{\prime}\cdot \nabla \widetilde{f}_+\right)}_{ = 0 \,\,\, 
       \mbox{(Lem.\ref{lm:smxprod})}} 
   d^2\vect{x}
   \Big\}^2 
   +
   \Big\{ \int_{\R^2} 
      \underbrace{Sc\left(\vect{a}\cdot \vect{x}f_+ \, \vect{b}\cdot \nabla \widetilde{f}_-
   \right)}_{ = 0 \,\,\, \mbox{(Lem.\ref{lm:smxprod})}} d^2\vect{x}
   \Big\}^2 
   \Big]
   \\
   = \,\,\, &(2\pi)^2\left[
   \left\{ \int_{\R^2} 
      \vect{a}\cdot \vect{x} \frac{1}{2}\vect{b}\cdot \nabla (f_- \widetilde{f}_-) d^2\vect{x}
   \right\}^2 
   +
   \left\{ \int_{\R^2} 
      \vect{a}\cdot \vect{x} \frac{1}{2}\vect{b}^{\prime}\cdot \nabla (f_+ \widetilde{f}_+) 
      d^2\vect{x}
      \right\}^2  
   \right] 
\end{align*}
\begin{align}
   \stackrel{\mbox{Int. by p.}}{=}
   &\frac{(2\pi)^2}{4}
   \Big[
   \Big\{ \int_{\R^2} 
      \underbrace{\vect{b}\cdot \nabla(\vect{a}\cdot \vect{x})}_{= \vect{a}\cdot \vect{b}}    |f_-|^2
   d^2\vect{x}
   \Big\}^2 
   +
   \Big\{ \int_{\R^2} 
      \underbrace{\vect{b}^{\prime}\cdot \nabla(\vect{a}\cdot \vect{x})}_{= \vect{a}\cdot\vect{b}^{\prime}}  |f_+|^2 
      d^2\vect{x}
      \Big\}^2  
   \Big]
   \nonumber \\
   & \,\,\, =
   \frac{(2\pi)^2}{4}
   \Big[ \,
      (\vect{a}\cdot \vect{b})^2 \,\,F_-^2 
      + (\vect{a}\cdot\vect{b}^{\prime})^2 \,\,F_+^2 \,
   \Big] \, .
\end{align}
After using the Schwartz inequality of Lemma \ref{lm:Swineq} we further used
\begin{align}
Sc\left(f_- \, \vect{b}\cdot \nabla \widetilde{f}_-\right)
  = \frac{1}{2} \vect{b}\cdot \nabla \left(f_- \, \widetilde{f}_-\right),
  \,\,\,
  Sc\left(f_+ \, \vect{b}\cdot \nabla \widetilde{f}_+\right)
  = \frac{1}{2} \vect{b}\cdot \nabla \left(f_+ \, \widetilde{f}_+\right).
\end{align}
\end{proof}

\section{Generalization of directional uncertainty principle to spacetime}

The quaternions frequently appear as subalgebras of higher order
Clifford geometric algebras~\cite{PL:Clifford}. This is for example the case for the
spacetime algebra (STA)~\cite{DH:STA}, which is of prime importance in 
physics, and in applications where time matters as well 
(motion in time, video sequences, flow fields, ...). 
The quaternion subalgebra allows to introduce generalizations
of the QFT to functions in these higher order Clifford geometric algebras.
For example it allows to generalize the QFT to a spacetime FT~\cite{EH:QFTgen}.

\subsection{Spacetime algebra (STA)}

The spacetime algebra is the geometric algebra of $\R^{3,1}$. 
In $\R^{3,1}$ we can introduce the following orthonormal basis
\begin{equation}
  \{\vect{e}_t,\vect{e}_1, \vect{e}_2, \vect{e}_3  \},
  \quad
  -\vect{e}_t^2 = \vect{e}_1^2 = \vect{e}_2^2 = \vect{e}_3^2 = 1.
\end{equation}
In $\mathcal{G}(\R^{3,1})$ we thus get three anti-commuting 
blades that all square to minus one
\begin{equation}
  \vect{e}_t^2 = -1, 
  \quad
  i_3 = \vect{e}_1\vect{e}_2\vect{e}_3, \quad i_3^2 = -1,
  \quad
  i_{st} = \vect{e}_t\vect{e}_1\vect{e}_2\vect{e}_3, 
  \quad i_{st}^2=-1.
\end{equation}
The volume-time subalgebra of $\mathcal{G}(\R^{3,1})$ generated by these blades is indeed
isomorphic to the quaternion algebra~\cite{PG:quat}. 
\begin{equation}
  \label{eq:isom}
  \{1, \vect{e}_t, i_3, i_{st} \} \longleftrightarrow \{1, \qi, \qj, \qk\}
\end{equation}
This isomorphism allows us now to introduce the  $\pm$ split to spacetime algebra,
which now turns out to be a very real (physical) spacetime split 
\begin{equation}
  f_{\pm} = \frac{1}{2}(f \pm \vect{e}_t f \vect{e}_t^{*}),
\label{eq:fpmsd}
\end{equation}
where
\be
  \vect{e}_t^{*} 
  = \vect{e}_t i_{st}^{-1} 
  = - \vect{e}_t i_{st} 
  = - \vect{e}_t \vect{e}_t i_{3} 
  = i_{3}.
\ee
The time direction $\vect{e}_t$ determines therefore the complimentary 3D space
with pseudoscalar $i_3$ as well!

\subsection{From the QFT to the spacetime Fourier transform (SFT)\label{sc:SFT}}
The spacetime Fourier transform
maps {16D} spacetime algebra functions 
$f : \R^{3,1} \rightarrow \G_{3,1}$
to {16D} spacetime spectrum functions
$\overset{\diamond}{f}: \R^{3,1} \rightarrow \G_{3,1}  $. It  is defined in the following way
\begin{equation}
  f \rightarrow \mathcal{F}_{SFT}\{f\}(\bomega) 
  = \overset{\diamond}{f}(\bomega) 
  = \int_{\R^{3,1}} e^{-\vect{e}_t \,t\omega_t} f(\vect{x}) \,
    e^{-i_3 \vec{x}\cdot \vec{\omega}} d^4\vect{x} \,,
  \label{eq:sftgen}
\end{equation}
with 
\begin{itemize}
\item
spacetime vectors
$\vect{x}= t\vect{e}_t + \vec{x} \in \R^{3,1}, \,\,
 \vec{x} = x\vect{e}_1+y\vect{e}_2+z\vect{e}_3 \in \R^{3} $
\item
spacetime volume $d^4\vect{x} = dt dx dy dz$
\item
spacetime frequency vectors
$\bomega= \omega_t\vect{e}_t + \vec{\omega} \in \R^{3,1}, \,\,
 \vec{\omega} = \omega_1\vect{e}_1+\omega_2\vect{e}_2+\omega_3\vect{e}_3 \in \R^{3} $ 
\end{itemize}
\begin{rem}
The 3D integration part 
$$
  \int f(\vect{x}) \,
  e^{-i_3 \vec{x}\cdot \vec{\omega}} d^3\vec{x}
$$
in \eqref{eq:sftgen} fully corresponds to the GA FT in $\G_{3}$, as stated in
\eqref{eq:3DCFT}, compare~\cite{MH:CFT,HM:ICCA7}.
\end{rem}

The $\pm$ split of the QFT can now, via the isomorphism \eqref{eq:isom} of
quaternions to the volume-time subalgebra of the spacetime algebra, be extended
to splitting general spacetime algebra multivector functions over $\R^{3,1}$.
This leads to the following interesting result~\cite{EH:QFTgen}.
\begin{gather}
  \overset{\diamond}{f} = 
  \overset{\diamond}{f}_+ + \overset{\diamond}{f}_-
  \stackrel{}{=} \int_{\R^{3,1}} f_{+}\,e^{-i_3 (\, \vec{x}\cdot \vec{\omega} -\, t\omega_t\,)} d^4x
               + \int_{\R^{3,1}} f_{-}\,e^{-i_3 (\, \vec{x}\cdot \vec{\omega} +\, t\omega_t\,)}  d^4x.
  \label{eq:SFTfpmrcomb}
\end{gather}
This result shows us that the SFT is identical to a sum of 
right and left propagating multivector wave packets.
We therefore see that these physically important wave packets arise
absolutely naturally from elementary purely algebraic considerations.

\subsection{Directional 4D spacetime uncertainty principle}

We now generalize the directional QFT uncertainty principle to spacetime algebra
multivector functions. 
\begin{thm}[Directional 4D spacetime uncertainty principle]
For two arbitrary constant spacetime vectors $\vect{a},\vect{b} \in \R^{3,1}$ (selecting two 
directions), and $f\in L^2(\R^{3,1},\G_{3,1})$, $|\vect{x}|^{1/2}f \in L^2(\R^{3,1},\G_{3,1})$
we obtain 
\begin{gather*}
 \int_{\R^{3,1}} (a_t t-\vec{a}\cdot \vec{x})^2 |f(\vect{x})|^2 d^4\vect{x}
 \int_{\R^{3,1}} (b_t \omega_t - \vec{b}\cdot \vec{\omega} )^2  |\mathcal{F}\{f\}(\bomega )|^2 d^4\bomega
 \\
 \geq \frac{(2\pi )^4}{4}
      \left[ (a_t b_t-\vec{a} \cdot \vec{b})^2 F_-^2  
             + (a_t b_t+\vec{a} \cdot \vec{b})^2 F_+^2 \right] ,
\end{gather*}
with energies of the left and right traveling wave packets:
\be
   F_{\pm} =  \int_{\R^2} |f_{\pm}(\vect{x})|^2 d^2 \vect{x}.
\ee
\end{thm}
\begin{rem}
The proof is strictly analogous to the proof of the directional QFT uncertainty principle. 
In that proof we simply need to replace the integral $\int_{\R^{2}}$ by the integral $\int_{\R^{3,1}}$,
the infinitesimal 2D area element $d^2\vect{x}$ by the infinitesimal 4D volume element $d^4\vect{x}$, etc.
Applying the Parseval theorem in four dimensions leads to the factor of $(2\pi)^4$ instead of $(2\pi)^2$.
\end{rem}

\section{Conclusion}

We first studied the $\pm$ split of quaternions and its effects
on the double sided QFT. Based on this
we have established a new \textit{directional} QFT uncertainty principle.

Via isomorphisms of the quaternion algebra $\mathbb{H}$ to Clifford
geometric subalgebras, the QFT can be generalized to these Clifford geometric algebras. 
This generalization to higher dimensional algebras is also
possible for the new directional QFT uncertainty principle. We
demonstrated this with the physically most relevant
generalization to spacetime algebra functions. 

There the quaternion $\pm$ split corresponds to the relativistic split of spacetime
into time and space, e.g. the rest frame of an observer. The split of the SFT corresponds then to analyzing a spacetime
multivector function in terms of left and right travelling multivector wave packets~\cite{EH:QFTgen}.
The energies of these wave packets determine together with two arbitrary spacetime directions
(one spacetime vector and one relativistic wave vector) the resulting uncertainty threshold. 

Regarding the proliferation of higher dimensional theories in theoretical physics,
and the model of Euclidean space in the socalled conformal geometric algebra $\G_{4,1}$,
which has recently become of interest for applications in
computer graphics, computer vision and geometric reasoning~\cite{DFM:GACS,HL:IAGR}, 
even higher dimensional generalizations of the directional QFT uncertainty
principle may be of interest in the future.

%






\subsection*{Acknowledgment}

\begin{quote}
I do believe in God's creative power in having brought it all into being in the first place, 
I find that studying the natural world 
is an opportunity to observe 
the majesty, the elegance, the intricacy 
of God's creation.~\cite{FC:Time}
\end{quote}

I thank my family for their constant loving support, as well as 
B. Mawardi and S. Buchholz.
I further thank the 
Cognitive Systems group (LS Sommer) in Kiel for their hospitality.

\end{document}